\documentclass[11pt]{article}
\usepackage{amsmath}
\usepackage{amssymb}
\usepackage{amsthm}
\usepackage{graphicx}
\usepackage{microtype}
\usepackage{hyperref}
\usepackage{enumitem}
\usepackage{mathrsfs}

\newtheorem{theorem}{Theorem}[section]
\newtheorem{lemma}[theorem]{Lemma}

\theoremstyle{definition}
\newtheorem{question}[theorem]{Question}

\newtheorem{definition}[theorem]{Definition}
\newtheorem{example}[theorem]{Example}

\newtheorem{remark}[theorem]{Remark}
\newtheorem{notation}[theorem]{Notation}

\def\x{\xi}
\def\p{\Phi}

\def\Z{\mathcal{Z}}

\def\U{L}
\def\au{l}

\def\S{\mathscr{S}}
\def\Z{\mathscr{Z}}
\def\nv{P_{S}}

\def\pro{\vartriangleleft_{x}}
\def\proy{\vartriangleleft_{y}}
\def\ppro{\blacktriangleleft_{x}}
\def\pproy{\blacktriangleleft_{y}}
\def\O{B}
\def\p{\phi^{n}}
\def\cut{28}

\makeatletter
\newcommand{\tpmod}[1]{{\@displayfalse\pmod{#1}}}
\makeatother

\makeatletter

\newcommand{\Rmnum}[1]{\expandafter\@slowromancap\romannumeral #1@}
\makeatother

\begin{document}
\title {\textbf{Distances and intersections of curves}}
\author{Yohsuke Watanabe}
\maketitle

\begin{abstract}
We obtain a coarse relationship between geometric intersection numbers of curves and the sum of their subsurface projection distances with explicit quasi-constants. By using this relationship, we give applications in the studies of the curve graphs and the mapping class groups.
\end{abstract}



\section{Introduction}
We let $S_{g,b}$ denote a genus $g$ compact surface with $b$ boundary components. We work on $S_{g,b}$ such that $\xi(S_{g,b})=3g+b-3>0$. We let $S$ denote a such surface when its topological type is not emphasized. Whenever we work with curves, arcs, and subsurfaces on $S$, we assume they and their boundary components are at the geodesic representatives with a hyperbolic structure which we sometimes specify. The curve graph $C(S)$ is a $1$-dimensional CW complex whose vertices are simple closed curves of $S$ and whose edges are realized between vertices if the corresponding curves in $S$ realize the minimal possible geometric intersection number in $S$. With the graph metric, $C(S)$ is a path-connected and infinite diameter space. Let $x,y\in C(S)$. We let $i(x,y)$ and $d_{S}(x,y)$ respectively denote the geometric intersection number of $x$ and $y$ in $S$ and the distance between $x$ and $y$ in $C(S)$. Those are fundamental measures associated to curves on surfaces, and they are the central objects in this paper. A classical relationship between them, that are due to Lickorish \cite{LIC} and Hempel \cite{HEM}, says that if $x,y\in C(S)$ then $$d_{S}(x,y)\leq 2 \log_{2} i(x,y)+2.$$ However, as one easily sees, this inequality is not sharp:  
\begin{example}\label{gap}
Let $x,a\in C(S)$ such that they intersect and are contained in a proper subsurface $Z$ of $S$. Let $T_{a}$ denote Dehn twist along $a$. Hempel--Lickorish inequality clearly holds for a pair, $x$ and $T_{a}^{n}(x)$, as $d_{S}(x, T_{a}^{n}(x))=2$ and $i(x, T_{a}^{n}(x))=|n| \cdot i(a,x)^{2}.$
However, the gap in the inequality is large if $|n|$ is large. 
Their subsurface projections distances are small in the subsurfaces that are not contained in $Z$, but their annular subsurface projection distance along $a$ grows linearly with $|n|$, so by adding it to the distance-side of Hempel--Lickorish inequality, some gap gets filled up, but if $i(a,x)$ is large, the annular subsurface projection distance is not enough. Nevertheless, it is not straightforward to detect other large subsurface projections in $Z$ except for the annulus. 
There could be enough many large subsurface projections between them so that they fill up the gap completely, or there could be too many large subsurface projections between them so that the distance-side exceeds the intersection-side in Hempel--Lickorish inequality. \end{example}

In the above light, the following theorem due to Choi--Rafi is powerful.

\begin{theorem}[\cite{CR}]\label{choirafi} 
There exists $k$ such that for any markings $\mu_{1}$ and $\mu_{2}$ on $S$, $$\log i(\mu_{1},\mu_{2})\asymp \sum_{Z\subseteq S}[ d_{Z}(\mu_{1},\mu_{2})]_{k}+\sum_{A\subseteq S} \log [d_{A}(\mu_{1},\mu_{2})]_{k}$$ where $A\asymp B$ if there exists $K^{\pm}$ and $C^{\pm}$ such that $\frac{B}{K^{-}}-C^{-}\leq A\leq K^{+}\cdot B+C^{+}$, where $[m]_{n}=m$ if $m>n$ and $[ m]_{n}=0$ if $m\leq n$, and where the sum is taken over all $Z$ which are not annuli and $A$ which are annuli.
\end{theorem}

Theorem \ref{choirafi} holds for curves, for instance see \cite{W4}. (Conversely, if one knows for curves then it implies for markings, which is obvious.) 
Therefore, $\log$ of geometric intersection number of two curves and the sum of their large subsurface projection distances are coarsely equal. 
The rough idea of the proof of Theorem \ref{choirafi} is as follows: they take $\sigma_{1}$ and $\sigma_{2}$ from the thick part of Teichum\"uller space so that $\mu_{1}$ and $\mu_{2}$ are respectively short on $\sigma_{1}$ and $\sigma_{2}$. They show that the Teichum\"uller distance between $\mu_{1}$ and $\mu_{2}$ and $\log i(\sigma_{1},\sigma_{2})$ are coarsely equal. Then they apply a result of Rafi \cite{RAF}, which says that the Teichum\"uller distance between two points in the thick part of Teichum\"uller space and the sum of large subsurface projection distances between two corresponding short markings are coarsely equal.

Even though Theorem \ref{choirafi} gives fulfillment to the gap in Hempel--Lickorish inequality, simple questions still remain:
\begin{question}\label{compute}
What is a cut-off constant and what are additive/multiplicative constants exactly? How do they depend on the topology of underlying surfaces?
\end{question}

\emph{For the rest of this paper, we fix the base of $\log$ to be $2$ and treat $\log 0$ as $0$.}

A partial solution to Question \ref{compute} was given:
\begin{theorem}[\cite{W3}]\label{igd}
Let $x,y\in C(S)$. For all $k> 0$, we have $$\log i(x,y)\leq V_{S,k} \cdot \bigg( \sum_{Z\subseteq S}[ d_{Z}(x,y)]_{k}+\sum_{A\subseteq S} \log [d_{A}(x,y)]_{k}\bigg) + V_{S,k}$$ where $V_{S,k}=\big( M^{2}|\chi(S)| \cdot (k+\x(S)\cdot M)  \big)^{\x(S)+2}$ with $M=200.$ 
\end{theorem}

In this paper, we are interested in the other direction, which then completes a solution to Question \ref{compute}. We think that promoting a coarse relation to an effective one is important: Theorem \ref{igd} is used to give an algorithm to compute the distance between two points in the curve graph \cite{W3}. The algorithm is computable because cut-off and additive/multiplicative constants are computable. Furthermore, Theorem \ref{igd} is used in a work of Aougab--Biringer--Gaster \cite{ABG} where they give a computable algorithm to determine if a given finite graph is an induced subgraph of a given curve graph. Again, the effectiveness of Theorem \ref{igd} plays an important role therein.

\subsection{Results}\label{results}
We show

\begin{theorem}\label{mainone}
Suppose $\x(S)=1$. Let $x,y\in C(S)$. For all $k\geq 18$, we have $$ \frac{ \sum_{Z\subseteq S}[ d_{Z}(x,y)]_{k}+\sum_{A\subseteq S} \log [d_{A}(x,y)]_{k}}{2\au_{\lceil \frac{k+1}{2} \rceil}+1}  -\frac{1}{2\au_{\lceil \frac{k+1}{2} \rceil}+1} \leq \log i(x,y)$$
where $\au_{\lceil \frac{k+1}{2} \rceil}= \frac{\log ( \lceil \frac{k+1}{2} \rceil) }{\log (\lceil \frac{k+1}{2} \rceil-5)}$ (See Figure \ref{one}.).
\end{theorem}

\begin{figure}[p]
 \begin{center}
  \includegraphics[width=120mm]{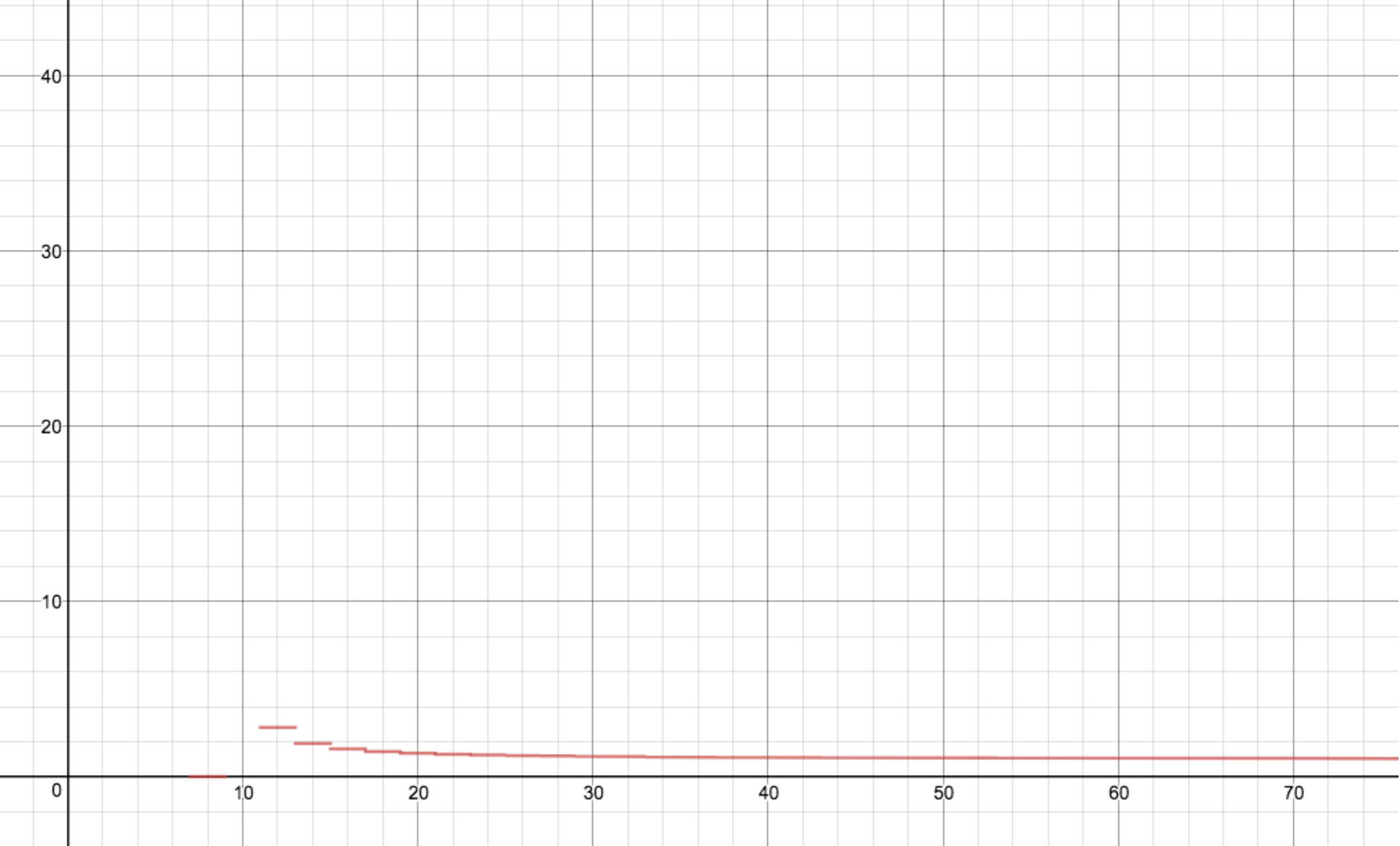}
 \end{center}
 \caption{The above function, $\au_{\lceil \frac{k+1}{2} \rceil}$, is bigger than $1$ and limits to $1$.}
 \label{one}
\end{figure}

We note that Theorem \ref{mainone} was proved by a different method in \cite{W3}. See Remark \ref{proofbybgit}. The cut-off constant was $600$ and the additive/multiplicative constants were $2$ in \cite{W3}.

We show
\begin{theorem}\label{main}
Suppose $\x(S)>1$. Let $x,y\in C(S)$. For all $k\geq  28$, we have $$ \frac{ \sum_{Z\subseteq S}[ d_{Z}(x,y)]_{k}+\sum_{A\subseteq S} \log [d_{A}(x,y)]_{k}}{U_{S,k}}  - \frac{2}{U_{S,k}}\leq \log i(x,y)$$ where $U_{S,k}= (\nv-1) \cdot  2\U_{\lceil \frac{k+1}{2} \rceil} +2$ with $\U_{\lceil \frac{k+1}{2} \rceil}= \frac{\lceil \frac{k+1}{2} \rceil}{\log \big(2^{\frac{\lceil \frac{k+1}{2} \rceil-12}{2}} -1\big)}$(See Figure \ref{two}.) and $P_{S}$ which is given by Lemma \ref{partition} (See Remark \ref{island}. In particular, $P_{S}\leq \xi(S)\cdot 4^{\xi(S)}$.).
\end{theorem}


\begin{figure}[p]
 \begin{center}
  \includegraphics[width=120mm]{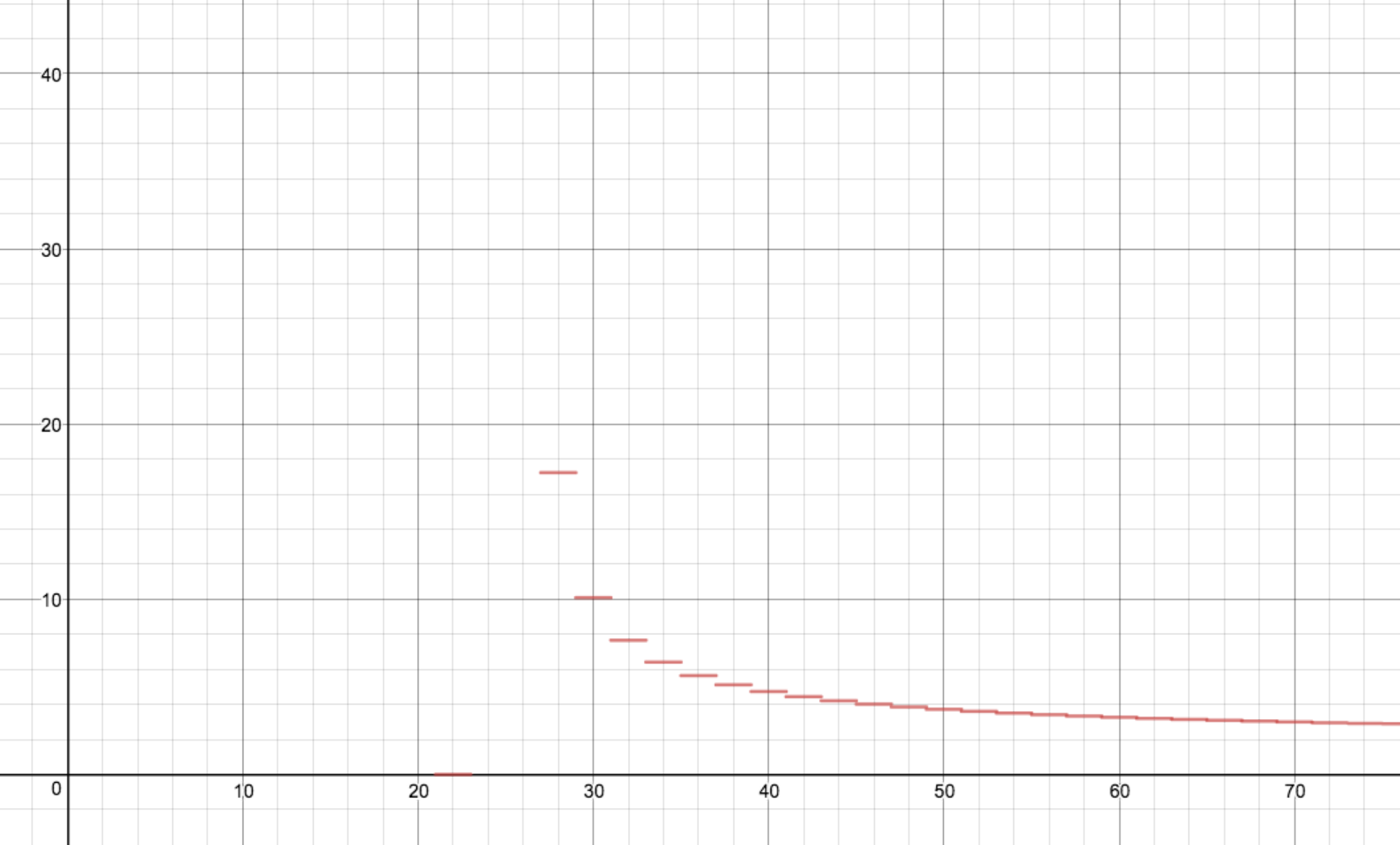}
 \end{center}
 \caption{The above function, $\U_{\lceil \frac{k+1}{2} \rceil}$, is bigger than $2$ and limits to $2$.}
 \label{two}
\end{figure}

Hence, combining with Theorem \ref{igd}, we have an effective coarse relationship between geometric intersection numbers of curves and the sum of their subsurface projection distances. By using this relationship, we give the following applications. 

We study the intersection numbers of the curves contained in geodesics in the curve graph. 
An uneffective and weaker version of it was given in \cite{W4}, but now with our effective relationship and a slightly more advanced technique, we obtain a stronger result. 
First, we study on Masur--Minsky tight geodesics. We show

\begin{theorem}
Let $x,y\in C(S)$ and $g=\{v_{i}\}$ be a tight geodesic between $x$ and $y$ such that $d_{S}(x,v_{i})=i$ for all $i$. Let $\{v_{t_{j}}\}_{j=0}^{n}\subseteq g$ such that $d_{S}(x, v_{t_{j}})+2<d_{S}(x, v_{t_{j+1}})$ for all $j$. For all $k\geq 228,$ we have $$\frac{\sum_{j=0}^{n-1}\log i(v_{t_{j}}, v_{t_{j+1}})}{V_{S,k}} \leq \bigg(\frac{k\cdot U_{S,k-200}}{k-200} +\frac{2}{3}\bigg)\cdot \log i(x,y)+  \frac{2k}{k-200} +\frac{2}{3} .$$ 
\end{theorem}        

Furthermore, we show a converse to the above for any geodesic.
\begin{theorem}
Let $x,y\in C(S)$ and $g=\{v_{i}\}$ be a geodesic between $x$ and $y$ such that $d_{S}(x,v_{i})=i$ for all $i$. Let $v_{p}, v_{q}\in g$ such that $d_{S}(x, v_{p})+2<d_{S}(x, v_{q})$. For all $k\geq 128 $, we have $$\frac{\log i(x,y)}{V_{S,k}}-1 \leq  \frac{k\cdot U_{S,k-100}}{k-100} \cdot \bigg( \log i(x,v_{q})+\log i(v_{p},y) \bigg) + \frac{4k}{k-100}.$$
\end{theorem}

We study intersection numbers of curves under iteration of pure mapping classes. 
In example \ref{gap}, we reviewed that if $x,a\in C(S)$ and $T_{a}$ is Dehn twist along $a$, then $i(x, T_{a}^{n}(x))=|n| \cdot i(a,x)^{2}.$ In the language of subsurface projections, it is
$$d_{A}(x, T_{a}^{n}(x) )+ 2\log i(x, A)-1= \log i(x, T_{a}^{n}(x))$$
where $A$ is the support of $T_{a}$, $Supp(T_{a})$. Dehn twists are the simplest pure mapping classes. We give an analogous identity for any pure mapping classes. We recall that if $\phi$ is a pure mapping class, then $\phi$ acts hyperbolically on the curve complexes of its supports, $Supp(\phi)$. 
We show

\begin{theorem}
Let $x \in C(S).$ Let $\phi$ be a pure mapping class. If $n$ is raised so that $d_{Z_{i}}(x,\p(x))>k\geq14$ for all $Z_{i}\in Supp(\phi),$ then 
$$\frac{  \sum_{Z_{i}\in Supp(\phi)} D_{Z_{i}}(x,\p(x)) +  2\cdot \big( \sum _{i}  \log i(x,Z_{i})  \big)}{\x(S)\cdot \U_{k+1} } \leq  \log i(x,\p(x)).$$
where $\U_{k+1}= \frac{k+1}{\log \big(2^{\frac{k-11}{2}} -1\big)}$ (See Figure \ref{four}) and where $D_{Z_{i}}$ denotes $\log d_{Z_{i}}$ if $Z_{i}$ is an annulus and $D_{Z_{i}}$ denotes $d_{Z_{i}}$ if $Z_{i}$ is not an annulus. 
\end{theorem}

We show
\begin{theorem}
Let $x \in C(S).$ Let $\phi$ be a pure mapping class. If $n$ is raised so that $d_{Z_{i}}(x,\p(x))>k\geq \max\{M_{\phi}, 32\}$ for all $Z_{i}\in Supp(\phi),$ then $$\frac{ \log i(x,\p(x)) }{V_{S,k}}-1\leq  \sum_{Z_{i}\in Supp(\phi)} D_{Z_{i}}(x,\p(x))+  \frac{2k\cdot U_{S,k-4}}{k-4} \cdot  \bigg(\sum _{i}  \log i(x,Z_{i}) \bigg)+ \frac{4k\cdot \xi(S)}{k-4}$$
where $M_{\phi}$ is such that $d_{W}(x, \phi^{n}(x))\leq M_{\phi}$ for all $x\in C(S)$, $n\in \mathbb{Z}$, and $W$ which is properly contained in the elements of $Supp(\phi)$ (See Section \ref{five}.) and where $D_{Z_{i}}$ denotes $\log d_{Z_{i}}$ if $Z_{i}$ is an annulus and $D_{Z_{i}}$ denotes $d_{Z_{i}}$ if $Z_{i}$ is not an annulus.
\end{theorem}

\subsection{Acknowledgements}
The author thanks Mladen Bestvina and Ken Bromberg for suggesting to obtain the main coarse relationship in this paper. The author thanks Tarik Aougab, Martin Bobb, Ken Bromberg, Hidetoshi Masai, and Nicholas Vlamis for useful conversations. 

The author thanks Sadayoshi Kojima for his hospitality while the author was in Tokyo Institute of Technology in Summer $2017$. The trip was beneficial to this paper. The author thanks his home institute, the University of Hawaii at Manoa, for supporting the trip. 

Finally, the author thanks U.S. National Science Foundation grants DMS 1107452, 1107263, 1107367 ``RNMS: GEometric structures And Representation varieties" (the GEAR Network).




\section{Preliminaries}

We start with the definition of subsurface projections from \cite{MM2}. 
\begin{definition}[\cite{MM2}] 
Let $x\in C(S).$
\begin{itemize}
\item Let $Z\subseteq S$ such that $Z$ is not an annulus, then $\pi_{Z}:C(S)\rightarrow C(Z)$ is defined as follows: $\pi_{Z}(x)$ is the set of all curves which arise as the boundary components of the regular neighborhoods of $a\cup Z$ for all $a\in \{x\cap Z\}.$
\item Let $A\subseteq S$ such that $A$ is an annulus. Let $A_{S}$ denote the compactification of $\mathbb{H}^{2}/\pi_{1}(A)$ and $p_{A}$ denote the associated covering map. The vertices of $C(A)$ are simple arcs which connects two boundary components of $A_{S}$. 
The edges are realized by disjointness. Then, $\pi_{Z}:C(S)\rightarrow C(A)$ is defined as follows: $\pi_{A}(x)$ is the set of all arcs which arise as the lift of $x$ via $p_{A}$. 
\end{itemize}
\end{definition}

\begin{notation}
Let $X,Y\subseteq S$. We let $d_{X}(Y,)$ denote $d_{X}(\partial(Y),)$ and let $i(X,)$ denote $i(\partial(X),).$
\end{notation}

The following is due to Behrstock \cite{BEH}. In this paper we use an effective one due to Leininger.

\begin{theorem}[Leininger] \label{BI}
Let $X,Y\subseteq S$ such that they overlap i.e. $X \pitchfork Y$. Let $\mu$ be a simplex. If $d_{X}(Y,\mu)>9$ then $d_{Y}(X,\mu)\leq 4.$
\end{theorem}

Theorem \ref{BI} plays a key role in the proofs of Theorem \ref{mainone} and Theorem \ref{main}.


The following is due to Masur--Minsky \cite{MM2}. In this paper we use the effective one due to Webb \cite{WEB1}. We do not use it in the proofs of Theorem \ref{mainone} and Theorem \ref{main}, but use it in Section \ref{app}.

\begin{theorem}[\cite{WEB1}] \label{BGIT}
Let $Z\subsetneq S$. Let $g$ be a geodesic in $C(S)$ such that every vertex of it projects nontrivially to $Z$. Then the diameter of $\pi_{Z}(g)$ is bounded by $\O\leq 100.$
\end{theorem}

\begin{remark}\label{proofbybgit}
In \cite{W3}, Theorem \ref{BGIT} played a key role in the proofs of Theorem \ref{igd} and Theorem \ref{mainone}. Theorem \ref{BGIT} says that if there is a large subsurface projection then the boundary components of the subsurface appears in $1$-neighborhood of a geodesic. It seems natural to use Theorem \ref{BGIT} to obtain the main coarse equality because we can detect all subsurfaces with large subsurface projection distances and it is easy to control them as their boundary components appear near a geodesic. In fact, if $\xi(S)=1$, their boundary components appear exactly on a geodesic, which is stronger information used to prove Theorem \ref{mainone} in \cite{W3}. However, if $\xi(S)>1$, we do not have this. Moreover, we may get many subsurfaces which miss a common vertex on a geodesic; they may be disjoint, nested, or may overlap. 
We give a new approach to prove Theorem \ref{mainone} using Theorem \ref{BI}. This approach applies to prove Theorem \ref{main}.
\end{remark}

\section{The main coarse relationship}\label{maintheorem}
The proof for $\xi(S)=1$ easily applies to the proof for $\xi(S)>1$ except for the complexity in dealing with more varieties of subsurfaces appearing in the coarse equality. We resolve this issue by Lemma \ref{partition}. We emphasize that the main idea between the case of $\xi(S)=1$ and the case of $\xi(S)>1$ is the same, so some parts in the proof for $\xi(S)>1$ is omitted.

\begin{notation}
Let $x,y\in C(S)$. Let $n>0$. We let $\Z(x,y,n)$ denote $\{Z\subsetneq S| d_{Z}(x,y)> n \}.$ 
\end{notation}

\begin{lemma}\label{annulus}
Let $x,y\in C(S)$. Let $A \in \Z(x,y,n)$ such that $A$ is an annulus. If $n\geq 6$, then $$\frac{\log d_{A}(x,y)}{\au_{n+1}} +\log i(x,A)+\log i(A,y)\leq \log i(x,y)$$ where $\au_{p}= \frac{\log p}{\log (p-5)}$ (See Figure \ref{three}.).

\end{lemma} 
\begin{proof}
In \cite{ABG}, they show the existence of a hyperbolic metric on $S$ so that if $a\in  \{x \cap A\}$ and $b\in  \{y \cap A\}$ give rise to $x'$ and $y'$ respectively via $\pi_{A}$, then $ i_{S_{A}}(x',y')\leq i_{A}(a,b)+2$ where $i_{D}$ denotes the geometric intersection number restricted in a domain $D$. Since $d_{A}(x,y)\leq i_{S_{A}}(x',y')+3$, we have $$d_{A}(x,y)-5\leq i_{A}(a,b)$$ for all $a\in  \{x \cap A\}$ and $b\in  \{y \cap A\}$. 
Clearly, $$\log d_{A}(x,y) \leq \frac{\log d_{A}(x,y)}{\log (d_{A}(x,y)-5)} \log i_{A}(a,b)$$ for all $a\in  \{x \cap A\}$ and $b\in  \{y \cap A\}$. 

By defining the following decreasing function $$\au_{p}= \frac{\log p}{\log (p-5)},$$ we have 
$$\log d_{A}(x,y)\leq \au_{n+1}  \log i_{A}(a,b)$$for all $a\in  \{x \cap A\}$ and $b\in  \{y \cap A\}$.

\begin{figure}[htbp]
 \begin{center}
  \includegraphics[width=118mm]{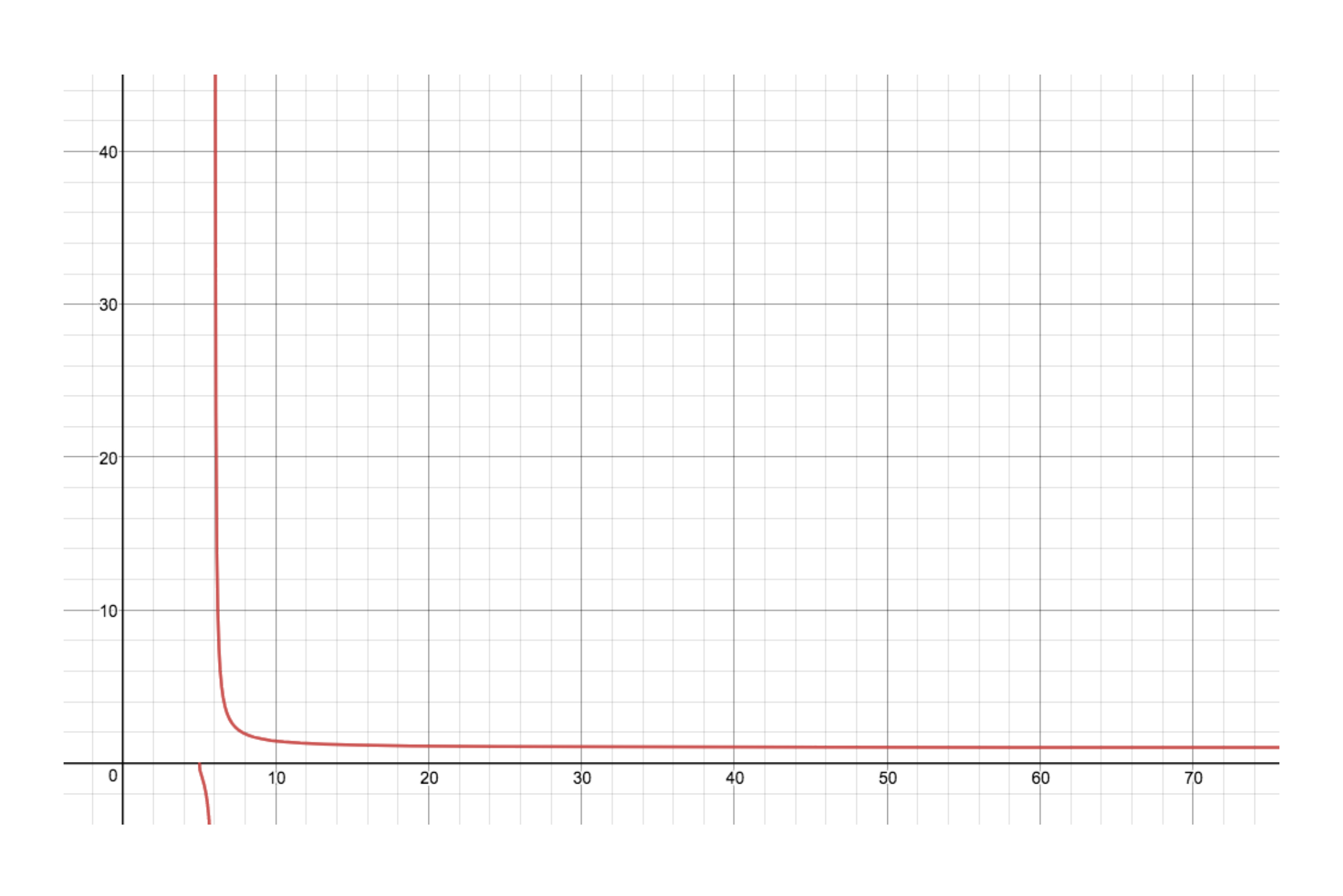}
 \end{center}
 \caption{$\au_{p}= \frac{\log p}{\log (p-5)}$}
 \label{three}
\end{figure}


Since $$\sum_{a\in  \{x \cap A\}, b\in  \{y \cap A\}}i_{A}(a,b)\leq i(x,y),$$ by using the fact that $\frac{\log d_{A}(x,y)}{\au_{n+1} }\leq  \log i_{A}(a,b)$ for all $a\in  \{x \cap A\}$ and $b\in  \{y \cap A\}$, we are done.



\end{proof}


\subsection{The proof for $\xi(S)=1$}
Throughout, we take an advantage of the fact that any two proper subsurfaces of $S$ overlap when $\xi(S)=1$. In particular, we use Theorem \ref{BI} often.

\begin{notation}\label{notationx}
Suppose $\x(S)=1$. Let $x,y\in C(S)$. Let $A_{p},A_{q}\in \Z(x,y,n)$. 
We let $A_{p} \pro A_{q}$ denote the following relation: $$\frac{\log d_{A_{q}}(x,y)}{2\au_{\lceil \frac{n+1}{2} \rceil}}+\log i(x,A_{q})\leq \log i(x,A_{p}).$$
\end{notation}

We have 
\begin{lemma}\label{replaceannulus}
Suppose $\x(S)=1$. Let $x,y\in C(S)$. If $n\geq 18$, then $A_{q} \pro A_{p}$ or $A_{p} \pro A_{q}$ for all $A_{p} ,A_{q}\in \Z(x,y,n)$. 
\end{lemma}
\begin{proof}
We have $d_{A_{p}}(x,y)\leq d_{A_{p}}(x,A_{q})+d_{A_{p}}(A_{q},y),$ so $\frac{d_{A_{p}}(x,y)}{2}\leq  d_{A_{p}}(x,A_{q})$ or $\frac{d_{A_{p}}(x,y)}{2}\leq  d_{A_{p}}(A_{q},y)$.

If $\frac{d_{A_{p}}(x,y)}{2}\leq  d_{A_{p}}(x,A_{q})$, then $6< \lceil \frac{n+1}{2} \rceil\leq d_{A_{p}}(x,A_{q})$. By Lemma \ref{annulus} we have $$\frac{\log d_{A_{p}}(x,A_{q})}{\au_{\lceil \frac{n+1}{2} \rceil}  }+\log i(x,A_{p})+\log i(A_{p}, A_{q}) \leq \log i(x,A_{q}).$$ Since $\frac{\log d_{A_{p}}(x,y)}{2}\leq \log d_{A_{p}}(x,y) -1\leq \log d_{A_{p}}(x,A_{q}),$ we have $$A_{q} \pro A_{p} .$$

If $\frac{d_{A_{p}}(x,y)}{2}\leq  d_{A_{p}}(A_{q},y)$, then $9<  d_{A_{p}}(A_{q},y)$. By Theorem \ref{BI} we have $ d_{A_{q}}(A_{p},y)\leq 4.$ Since $d_{A_{q}}(x,y)\leq d_{A_{q}}(x,A_{p})+d_{A_{q}}(A_{p},y)$, we have $6< \lceil \frac{n+1}{2} \rceil <d_{A_{q}}(x,y) -4 \leq d_{A_{q}}(x,A_{p})$. By Lemma \ref{annulus} we have $$\frac{\log d_{A_{q}}(x,A_{p})}{\au_{\lceil \frac{n+1}{2} \rceil}  }+\log i(x,A_{q})+\log i(A_{q}, A_{p}) \leq \log i(x,A_{p}).$$
Since $\frac{\log d_{A_{q}}(x,y)}{2}\leq  \log (d_{A_{q}}(x,y)-4)  \leq    \log d_{A_{q}}(x,A_{p}),$ we have $$A_{p} \pro A_{q} .$$

\end{proof}

The following is the heart of the proof of Theorem \ref{mainannulus}. 

\begin{theorem}\label{overlappingannulus}
Suppose $\x(S)=1$. Let $x,y\in C(S)$. For all $n\geq18$, we have $$\frac{\sum_{A \in \Z(x,y,n)} \log d_{A} (x,y)}{2\au_{\lceil \frac{n+1}{2} \rceil}} \leq \log i(x,y).$$
\end{theorem} 
\begin{proof}

We prove the above by an induction on $|\Z(x,y,n)|.$ Suppose $\Z(x,y,n)=\{A_{i}\}_{i=1}^{N}.$
Pick $A_{1}\in \Z(x,y,n)$. By Lemma \ref{annulus}, we have $$\frac{\log d_{A_{1}}(x,y)}{\au_{n+1}} +\log i(x,A_{1})+\log i(A_{1},y)\leq \log i(x,y).$$ 
Since $\au_{n+1}\leq 2\au_{\lceil \frac{n+1}{2} \rceil}$, we have $$(\dagger): \frac{\log d_{A_{1}}(x,y)}{2\au_{\lceil \frac{n+1}{2} \rceil}} +\log i(x,A_{1})+\log i(A_{1},y)\leq \log i(x,y).$$ 
We have $d_{A_{i}}(x,y)\leq d_{A_{i}}(x,A_{1})+ d_{A_{i}}(A_{1},y)$ for all $i.$ Hence we have $\frac{d_{A_{i}}(x,y)}{2}\leq d_{A_{i}}(x,A_{1})$ or $\frac{d_{A_{i}}(x,y)}{2}\leq d_{A_{i}}(A_{1},y)$ for all $i$. We assume $\frac{d_{A_{i}}(x,y)}{2}\leq d_{A_{i}}(x,A_{1})$ for all $i$, then, by Lemma \ref{replaceannulus}, we have $A_{1}\pro A_{i}$ for all $i$. (The argument given with this assumption is enough. Throughout, we keep $\log i(A_{1},y)$ in $(\dagger)$ unchanged. See Remark \ref{fory}.) 
It suffices to show $$A_{1}\pro A_{2}\pro \cdots \pro A_{N}$$ after re-indexing the elements of $\Z(x,y,n)$ fixing the original $A_{1}$. 

Pick $N-1$ elements from $\Z(x,y,n)$ including $A_{1}$ and re-index those fixing the original $A_{1}$ so that we have $$A_{1}\pro A_{2}\pro \cdots \pro A_{N-1}.$$ 
We let $A_{N}$ re-denote the unpicked element. By Lemma \ref{replaceannulus}, we have $A_{N-1}\pro A_{N}$ or $A_{N}\pro A_{N-1}$. In the former case, we are done. In the latter case, we proceed to the same argument;  by Lemma \ref{replaceannulus} we have $A_{N-2}\pro A_{N}$ or $A_{N}\pro A_{N-2}$. In the former case, we are done. In the latter case, we proceed to the same argument. We iterate this process, which eventually terminates since $A_{1}\pro A_{N}.$ 
\end{proof}

\begin{remark}\label{fory}
In the proof of Theorem \ref{overlappingannulus}, we assumed that $\frac{d_{A_{i}}(x,y)}{2}\leq d_{A_{i}}(x,A_{1})$ for all $i$ and constructed a chain which starts from $A_{1}$ via the binary relation, $\pro$, to replace $\log i(x,A_{1})$ in $(\dagger)$. Throughout the process, we left $\log i(A_{1},y)$ in $(\dagger)$ unchanged. However, in general, we may not have $\frac{d_{A_{i}}(x,y)}{2}\leq d_{A_{i}}(x,A_{1})$ for all $i$. For this, we consider the dual binary relation, $\proy$, given by swapping $x$ and $y$ in the inequality in Notation \ref{notationx}. Lemma \ref{replaceannulus} holds for $\proy$. 
Therefore, in general, after we pick $A_{1}$ and obtain $(\dagger)$, we partition $\{A_{i}\}_{i>1}^{N}$ into two collections, depending on $\frac{d_{A_{i}}(x,y)}{2}\leq d_{A_{i}}(x,A_{1})$ or $\frac{d_{A_{i}}(x,y)}{2}\leq d_{A_{i}}(A_{1},y)$. Then, in each collection, we separately run the same argument given in Theorem \ref{overlappingannulus} to construct an appropriate chain which starts from $A_{1}$, i.e., the former one via $\pro$ that serves to replace $\log i(x,A_{1})$ in $(\dagger)$ and the latter one via $\proy$ that serves to replace $\log i(A_{1},y)$ in $(\dagger)$. 
\end{remark}

\begin{theorem}\label{mainannulus}
Suppose $\x(S)=1$. Let $x,y\in C(S)$. For all $k\geq 18$, we have $$ \frac{ \sum_{Z\subseteq S}[ d_{Z}(x,y)]_{k}+\sum_{A\subseteq S} \log [d_{A}(x,y)]_{k}}{2\au_{\lceil \frac{k+1}{2} \rceil}+1}  -\frac{1}{2\au_{\lceil \frac{k+1}{2} \rceil}+1} \leq \log i(x,y).$$
\end{theorem}
\begin{proof} 
We observe $$ \sum_{Z\subseteq S}[ d_{Z}(x,y)]_{k}+\sum_{A\subseteq S} \log [d_{A}(x,y)]_{k}=[ d_{S}(x,y)]_{k}+\sum_{A \in \Z(x,y,k)} \log d_{A} (x,y).$$ By Hempel--Lickorish inequality for  $\x(S)=1$, we have $d_{S}(x,y)\leq  \log i(x,y)+1$. We are done with Theorem \ref{overlappingannulus}.
\end{proof}

\subsection{The proof for $\xi(S)>1$}

In the proof for $\xi(S)=1$, we separately proved the main theorem for $\{S\}$ and $\Z(x,y,n)$, taking an advantage of the fact that any two subsurfaces in the latter collection overlap. Then we combined them in Theorem \ref{mainannulus}. We essentially run the same argument. However, if $\xi(S)>1$, we may easily have two subsurfaces in $S$ such that they do not overlap. We overcome this issue by Lemma \ref{partition} which aids from the following result of Gaster--Greene--Vlamis \cite{GGV}. 

\begin{theorem}[\cite{GGV}]\label{color}
The chromatic number of $C(S_{g,b})$, that is the fewest number of colors required to color the vertices of $C(S_{g,b})$ so that adjacent vertices obtain different colors, is bounded by $g\cdot 4^{g}$ if $b=0$ and $(g+1)\cdot 2^{2g+b-1}$ if $b>0$.
\end{theorem}

In Theorem \ref{color}, the chromatic number of $C(S_{g,b})$ if $b=0$ is Theorem 1.1 of \cite{GGV}. We are grateful to Nicholas Vlamis for providing the chromatic number of $C(S_{g,b})$ if $b>0$, which, according to him, follows from the estimates in Section $5.5$ of \cite{GGV}. 

We have 

\begin{lemma}\label{partition}
There exists a partition of the set of all subsurfaces of $S$ into $\nv$ many collections so that any two subsurfaces in the same collection overlap, where $\nv $ is bounded by the chromatic number of $C(S)$.
\end{lemma} 
\begin{proof}
We give the coloring given by Theorem \ref{color} to all boundary components of all subsurfaces of $S$. Subsurfaces are in the same group if their boundary components have a color in common. This gives a desired partition. 
\end{proof}
We are grateful to Martin Bobb for conversation on the above.
\begin{remark}\label{island}
Theorem \ref{main} is expressed with $\nv$, Theorem \ref{main} holds for the minimum possible constant for $\nv$ which corresponds to the sharpest partition which satisfies the pairwise overlapping condition stated in Lemma \ref{partition}. For instance, the minimum constant for $\nv$ is $2$ if $\x(S)=1.$
\end{remark}

Now, we start the proof.

\begin{lemma}\label{base}
Let $x,y\in C(S)$. Let $Z \in \Z(x,y,n)$ such that $Z$ is not an annulus. If $n\geq 14$, then $$\frac{d_{Z}(x,y)}{\U_{n+1}}+\log i(x,Z)+\log i(Z,y) \leq \log i(x,y)$$ where $\U_{p}= \frac{p}{\log \big(2^{\frac{p-12}{2}} -1\big)}$ (See Figure \ref{four}.).

\end{lemma}
\begin{proof}
We have $d_{Z}(x,y)-6\leq d_{Z}(x',y')$ for all $x'\in \pi_{Z}(x)$, $y' \in \pi_{Z}(y)$ by the definition of $\pi_{Z}$.
If $a\in  \{x \cap Z\}$ and $b\in  \{y \cap Z\}$ give rise to $x'$ and $y'$ respectively via $\pi_{Z}$, then $i_{Z}(x',y')\leq 4 i_{Z}(a,b)+4$ by the definition of $\pi_{Z}$. 
With Hempel--Lickorish inequality, we have 
$$d_{Z}(x,y)-6\leq 2\log (4 i_{Z}(a,b)+4)+2$$ for all $a\in  \{x \cap Z\}$ and $b\in  \{y \cap Z\}$. Hence, we have $$ 2^{\frac{d_{Z}(x,y)-12}{2}}-1  \leq  i_{Z}(a,b)$$ for all $a\in  \{x \cap Z\}$ and $b\in  \{y \cap Z\}$. 
Clearly, $$ d_{Z}(x,y)\leq \frac{d_{Z}(x,y)}{\log \big(2^{\frac{d_{Z}(x,y)-12}{2}}-1\big )}  \log i_{Z}(a,b)$$ for all $a\in  \{x \cap Z\}$ and $b\in  \{y \cap Z\}$. 


By defining the following decreasing function $$\U_{p}= \frac{p}{\log \big(2^{\frac{p-12}{2}} -1\big)},$$ we have 
$$d_{Z}(x,y)\leq \U_{n+1}  \log i_{Z}(a,b)$$for all $a\in  \{x \cap Z\}$ and $b\in  \{y \cap Z\}$.


\begin{figure}[htbp]
 \begin{center}
  \includegraphics[width=118mm]{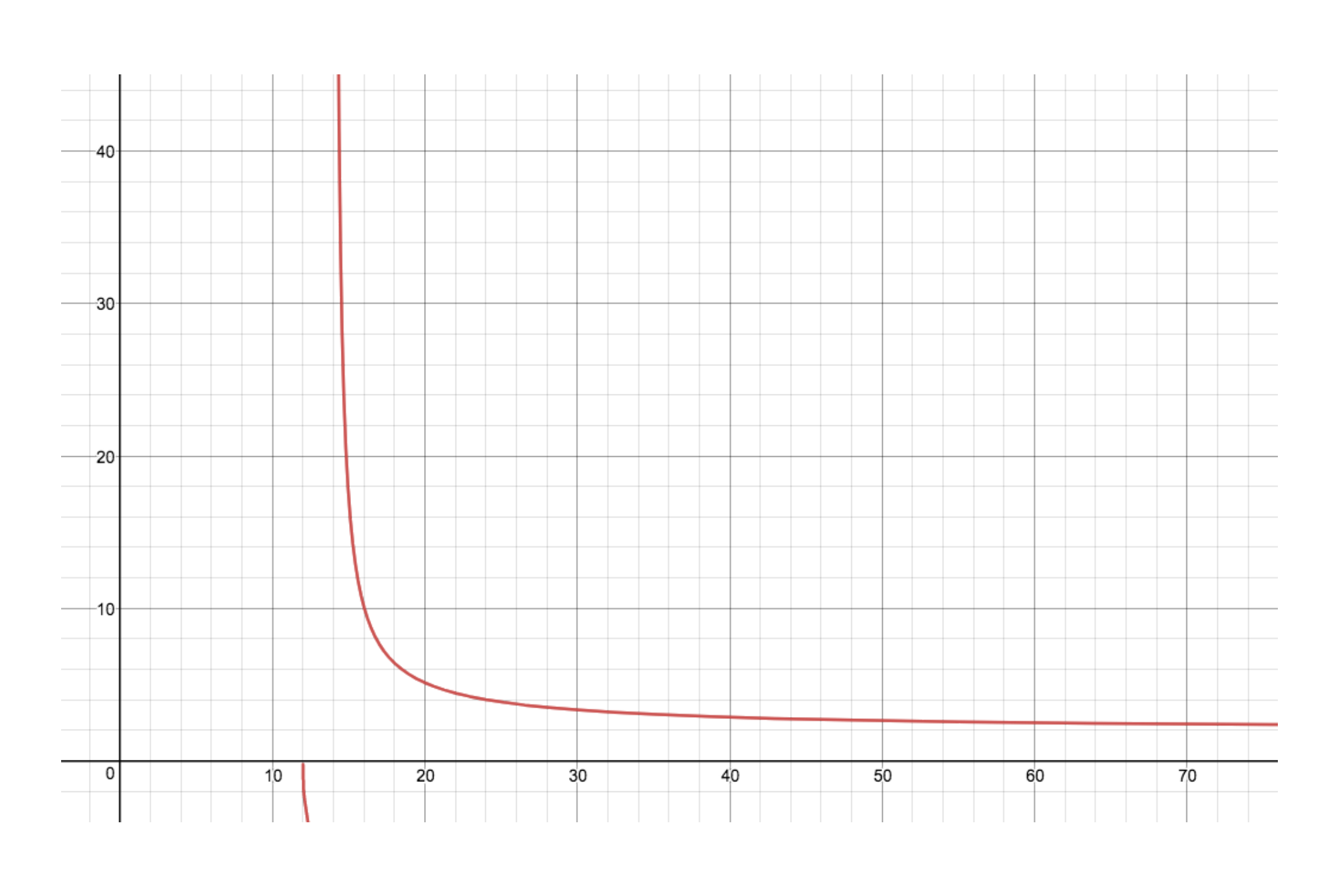}
 \end{center}
 \caption{$\U_{p}= \frac{p}{\log \big(2^{\frac{p-12}{2}} -1\big)}$}
  \label{four}
\end{figure}

As in the proof of Lemma \ref{annulus}, we have $$\sum_{a\in  \{x \cap Z\}, b\in  \{y \cap Z\}}i_{Z}(a,b)\leq i(x,y)$$ and we are done. 
\end{proof}

For simplicity, we use the following notation.

\begin{notation}
Let $x,y\in C(S)$. Let $Z\subseteq S$. We let $D_{Z}(x,y)$ denote $\log d_{Z}(x,y)$ if $Z$ is an annulus and let $D_{Z}(x,y)$ denote $d_{Z}(x,y)$ if $Z$ is not an annulus. 
\end{notation}

Now, we run the same argument given in the case of $\x(S)=1$, we omit details in the proofs.

\begin{notation}
Suppose $\x(S)>1$. Let $x,y\in C(S)$. Let $Z_{p},Z_{q}\in \Z(x,y,n)$ such that they overlap.
We let $Z_{p} \ppro Z_{q}$ denote the following relation: $$\frac{D_{Z_{q}}(x,y)}{2\U_{\lceil \frac{n+1}{2} \rceil}}+\log i(x,Z_{q})\leq \log i(x,Z_{p}).$$
\end{notation}

\begin{lemma}\label{replace}
Suppose $\x(S)>1$. Let $x,y\in C(S)$. If $n\geq 28$, then $Z_{q} \ppro Z_{p}$ or $Z_{p} \ppro Z_{q}$ for all $Z_{p} ,Z_{q}\in \Z(x,y,n)$ such that they overlap.
\end{lemma}
\begin{proof}
We have $d_{Z_{p}}(x,y)\leq d_{Z_{p}}(x,Z_{q})+d_{Z_{p}}(Z_{q},y),$ so $\frac{d_{Z_{p}}(x,y)}{2}\leq  d_{Z_{p}}(x,Z_{q})$ or $\frac{d_{Z_{p}}(x,y)}{2}\leq  d_{Z_{p}}(Z_{q},y)$.

If $\frac{d_{Z_{p}}(x,y)}{2}\leq  d_{Z_{p}}(x,Z_{q})$, then $14< d_{Z_{p}}(x,Z_{q})$. By Lemma \ref{annulus} when $Z_{p}$ is an annulus and Lemma \ref{base} when $Z_{p}$ is not an annulus, with the fact that $\au_{\lceil \frac{n+1}{2} \rceil}\leq \U_{\lceil \frac{n+1}{2} \rceil}$, we have $$\frac{D_{Z_{p}}(x,Z_{q})}{\U_{\lceil \frac{n+1}{2} \rceil}}+\log i(x,Z_{p})+\log i(Z_{p}, Z_{q}) \leq \log i(x,Z_{q})$$ so $Z_{q} \ppro Z_{p} .$

If $\frac{d_{Z_{p}}(x,y)}{2}\leq  d_{Z_{p}}(Z_{q},y)$, then by the same argument given in the proof of Lemma \ref{replaceannulus}, we have $14< d_{Z_{q}}(x,Z_{p})$. By the same argument given in the first case, we have $Z_{p} \ppro Z_{q}.$
\end{proof}

Lemma \ref{replace} is valid for the dual binary relation, $\pproy$, described in Remark \ref{fory}. We have 

\begin{theorem}\label{mainnonannulus}
Suppose $\x(S)>1$. Let $x,y\in C(S)$. For all $k\geq  28$, we have $$ \frac{ \sum_{Z\subseteq S}[ d_{Z}(x,y)]_{k}+\sum_{A\subseteq S} \log [d_{A}(x,y)]_{k}}{U_{S,k}}  - \frac{2}{U_{S,k}}\leq \log i(x,y)$$ where $U_{S,k}= (\nv-1) \cdot  2\U_{\lceil \frac{k+1}{2} \rceil} +2$.
\end{theorem}
\begin{proof} 
Let $\{\mathscr{P}_{i}\}_{i=1}^{\nv}$ be a partition of $\Z(x,y,k) \cup \{S\}$, which satisfies the pairwise overlapping condition stated in Lemma \ref{partition}. In particular, $\nv \leq \xi(S)\cdot 4^{\xi(S)}$. 
If $\mathscr{P}_{i}\neq \{S\}$, then we run the same argument given in Theorem \ref{overlappingannulus} using Lemma \ref{replace} with the fact that $\U_{k+1}\leq 2\U_{\lceil \frac{k+1}{2}\rceil}$:$$\frac{ \sum_{Z \in {\mathscr{P}_{i}}} D_{Z}(x,y)}{2\U_{\lceil \frac{k+1}{2} \rceil} }  \leq \log i(x,y).  $$ 
If $\mathscr{P}_{i}=\{S\}$, then we have Hempel--Lickorish inequality: $d_{S}(x,y)\leq 2 \log i(x,y)+2.$

We combine them:
$$ \frac{ \sum_{Z\subseteq S}[ d_{Z}(x,y)]_{k}+\sum_{A\subseteq S} \log [d_{A}(x,y)]_{k}}{(\nv-1) \cdot  2\U_{\lceil \frac{k+1}{2} \rceil} +2}  - \frac{1}{ (\nv-1) \cdot  \U_{\lceil \frac{k+1}{2} \rceil} +1}\leq \log i(x,y).$$

\end{proof}

\section{Applications}

\begin{notation}
Let $a,b\in C(S)$. Let $k>0$. We let $\S(a,b,k)$ denote $\sum_{Z\subseteq S}[ d_{Z}(a,b)]_{k}+\sum_{A\subseteq S} \log [d_{A}(a,b)]_{k}$. 
\end{notation}

We have established the following.
\begin{theorem}\label{th} 
Let $x,y\in C(S)$. For all $k\geq 28$, we have 
$$ \frac{  \log i(x,y) }{V_{S,k}} -1\leq  \S(x,y,k)  \leq  U_{S,k} \cdot \log i(x,y)+2$$ where $V_{S,k}$ and $U_{S,k}$ are given by Theorem \ref{igd} and Theorem \ref{main} respectively.

\end{theorem}

Now, we give applications using Theorem \ref{th}.

\subsection{On the curve graphs}\label{app}

We recall the constant $B\leq 100$ given by Theorem \ref{BGIT}.
\begin{theorem}
Let $x,y\in C(S)$ and $g=\{v_{i}\}$ be a tight geodesic between $x$ and $y$ such that $d_{S}(x,v_{i})=i$ for all $i$. Let $\{v_{t_{j}}\}_{j=0}^{n}\subseteq g$ such that $d_{S}(x, v_{t_{j}})+2<d_{S}(x, v_{t_{j+1}})$ for all $j$. For all $k\geq 200+\cut,$ we have $$\frac{\sum_{j=0}^{n-1}\log i(v_{t_{j}}, v_{t_{j+1}})}{V_{S,k}} \leq \bigg(\frac{k\cdot U_{S,k-200}}{k-200} +\frac{2}{3}\bigg)\cdot \log i(x,y)+  \frac{2k}{k-200} +\frac{2}{3} .$$ 
\end{theorem}                  
\begin{proof}
If $d_{W}(v_{t_{j}}, v_{t_{j+1}})>\O$, then  $d_{W}(x,v_{t_{j}})\leq \O$ and $d_{W}(v_{t_{j+1}},y)\leq \O$, see Lemma 2.4 of \cite{W4} whose proof uses Theorem \ref{BGIT} and tightness. Therefore, $d_{W}(v_{t_{j}}, v_{t_{j+1}})\leq d_{W}(x,y)+2\O\leq d_{W}(x,y)+200.$
We have
\begin{eqnarray*}
\sum_{j=0}^{n-1} \bigg(\frac{\log i(v_{t_{j}}, v_{t_{j+1}})}{V_{S,k}}-1\bigg)
&\leq& \sum_{j=0}^{n-1} \S(v_{t_{j}}, v_{t_{j+1}},k)
\\&\leq &\frac{k}{k-200}\cdot \S(x,y,k-200)
\\&\leq &\frac{k}{k-200}\cdot \bigg( U_{S,k-200} \cdot \log i(x,y)+2 \bigg)
\\&\leq &\frac{k\cdot U_{S,k-200}}{k-200} \cdot \log i(x,y)+\frac{2k}{k-200}.
\end{eqnarray*}
We are done by Hempel--Lickorish inequality; $n\leq \lfloor \frac{d_{S}(x,y)}{3} \rfloor\leq  \frac{2\log i(x,y)}{3} +\frac{2}{3} .$
\end{proof}


\begin{theorem}
Let $x,y\in C(S)$ and $g=\{v_{i}\}$ be a geodesic between $x$ and $y$ such that $d_{S}(x,v_{i})=i$ for all $i$. Let $v_{p}, v_{q}\in g$ such that $d_{S}(x, v_{p})+2<d_{S}(x, v_{q})$. For all $k\geq 100+ \cut $, we have $$\frac{\log i(x,y)}{V_{S,k}}-1 \leq  \frac{k\cdot U_{S,k-100}}{k-100} \cdot \bigg( \log i(x,v_{q})+\log i(v_{p},y) \bigg) + \frac{4k}{k-100}.$$
\end{theorem}
\begin{proof}
If $d_{W}(x,y)>\O$, then we have $\pi_{W}(v_{i})=\emptyset$ for some $i$ by Theorem \ref{BGIT}. We let $I$ denote $\min \{i| \pi_{W}(v_{i})=\emptyset\}.$
If $I>p$, then $\pi_{W}(v_{i})\neq \emptyset$ for all $i\leq p,$ so $d_{W}(x,v_{p})\leq \O$ by Theorem \ref{BGIT}. Hence, $d_{W}(x,y)\leq d_{W}(v_{p},y)+\O\leq d_{W}(v_{p},y)+100.$
If $I\leq p$, then $\pi_{W}(v_{i})\neq \emptyset$ for all $i\geq q,$ so $d_{W}(v_{q},y)\leq \O$ by Theorem \ref{BGIT}. Hence, $d_{W}(x,y)\leq d_{W}(x,v_{q})+\O \leq d_{W}(x,v_{q})+100.$
We have
\begin{eqnarray*}
\frac{\log i(x,y)}{V_{S,k}}-1 &\leq& \S(x,y,k)
\\&\leq &\frac{k}{k-100}\cdot \bigg( \S(x,v_{q},k-100)+  \S(v_{p},y,k-100) \bigg)
\\&\leq & \frac{k\cdot U_{S,k-100}}{k-100} \cdot \bigg( \log i(x,v_{q})+\log i(v_{p},y) \bigg) +  \frac{4k}{k-100}.
\end{eqnarray*}
\end{proof}

\subsection{On the mapping class groups}\label{five}
Let $\phi$ be a pure mapping class. Then $\phi$ acts hyperbolically on the curve complexes of its supports, $Supp(\phi)$. If $Z\in Supp(\phi)$ such that $Z$ is not an annulus, then $d_{Z}(x, \phi^{n}(x))\geq \frac{|n|}{200  |\chi(S)|^{2}}$ by Gadre--Tsai \cite{GadreTsai}. If $A\in Supp(\phi)$ such that $A$ is an annulus, then $d_{A}(x, \phi^{n}(x))\geq |n|$ by Masur--Minsky \cite{MM2}. Nevertheless, by a work of Minsky \cite{MIN}, if $W\subsetneq Z\in Supp(\phi)$, then there exists $M_{\phi}$ such that $d_{W}(x, \phi^{n}(x))\leq M_{\phi}$ for all $x\in C(S)$ and $n\in \mathbb{Z}.$ We are grateful to Tarik Aougab and Hidetoshi Masai for explaining this. 

\begin{theorem}\label{pure}
Let $x \in C(S).$ Let $\phi$ be a pure mapping class. If $n$ is raised so that $d_{Z_{i}}(x,\p(x))>k\geq14$ for all $Z_{i}\in Supp(\phi),$ then 
$$\frac{  \sum_{Z_{i}\in Supp(\phi)} D_{Z_{i}}(x,\p(x)) +  2\cdot \big( \sum _{i}  \log i(x,Z_{i})  \big)}{\x(S)\cdot \U_{k+1} } \leq  \log i(x,\p(x)).$$
\end{theorem}
\begin{proof}
By the argument given in Lemma \ref{annulus} and Lemma \ref{base}, with the fact that $\au_{k+1}\leq \U_{k+1}$, we have $$\frac{D_{Z_{i}}(x,y)}{\U_{k+1}}\leq \log i_{Z_{i}}(a,b)$$
 for all $a\in \{x\cap Z_{i}\}$ and $b\in \{ \p(x)\cap Z_{i}\}.$ 
Also, since $$\sum_{Z_{i}\in Supp(\phi)} \bigg( \sum_{a\in  \{x \cap Z_{i}\}, b\in  \{\p(x) \cap Z_{i}\}}i_{Z_{i}}(a,b) \bigg)\leq i(x,\p(x))$$
and $i(x, Z_{i})=i(\p(x), \p(Z_{i}))= i(\p(x), Z_{i})$, with AM-GM inequality, we have $$\frac{  \sum_{Z_{i}\in Supp(\phi)} D_{Z_{i}}(x,\p(x)) +  2\cdot \big( \sum _{i}  \log i(x,Z_{i})  \big)}{ \x(S)\cdot \U_{k+1} } \leq \log i(x,\p(x)).$$
\end{proof}

\begin{theorem}
Let $x \in C(S).$ Let $\phi$ be a pure mapping class. If $n$ is raised so that $d_{Z_{i}}(x,\p(x))>k\geq \max\{M_{\phi}, \cut+4\}$ for all $Z_{i}\in Supp(\phi),$ then $$\frac{ \log i(x,\p(x)) }{V_{S,k}}-1\leq  \sum_{Z_{i}\in Supp(\phi)} D_{Z_{i}}(x,\p(x))+  \frac{2k\cdot U_{S,k-4}}{k-4} \cdot  \bigg(\sum _{i}  \log i(x,Z_{i}) \bigg)+ \frac{4k\cdot \xi(S)}{k-4}.$$
\end{theorem}
\begin{proof}
If $W\subseteq  S\setminus Supp(\phi)$ then $d_{W}(x,\p(x))\leq 3.$ Since $k\geq M_{\phi}$, if $W\in  \{\Z(x,\p(x),k)\setminus Supp(\phi)\}$ then $W$ overlaps with some $Z_{i}\in Supp(\phi).$ 
Let 
\begin{itemize}
\item $Supp(\phi)^{c}$ denote $\{\Z(x,\p(x),k)\setminus Supp(\phi)\}.$ 
\item $Supp(\phi)^{c}_{i}$ denote $\{W\in Supp(\phi)^{c} | W \text{ overlaps with }Z_{i} \}.$ 
\end{itemize}
If $W\in Supp(\phi)^{c}_{i}$, then $d_{Z_{i}}(x,W)>9$ or $d_{Z_{i}}(W,\p(x))>9$ since $d_{Z_{i}}(x,\p(x))\leq d_{Z_{i}}(x,W)+d_{Z_{i}}(W,\p(x))$. Hence we have $d_{W}(x,Z_{i})\leq 4$ or $d_{W}(Z_{i},\p(x))\leq 4$ by Theorem \ref{BI}. Since $d_{W}(x,\p(x))\leq d_{W}(x,Z_{i})+d_{W}(Z_{i},\p(x))$, we have 
\begin{eqnarray*}
\sum_{W\in Supp(\phi)^{c}_{i}} D_{W}(x,\p(x)) 
&\leq & \frac{k}{k-4}\cdot \bigg( \sum_{W\in Supp(\phi)^{c}_{i}}  \big( D_{W}(x,Z_{i}) +D_{W}(Z_{i},\p(x))  \big)  \bigg)
\\&\leq & \frac{k}{k-4}\cdot \bigg( \S(x,Z_{i}, k-4) +\S(Z_{i},\p(x),k-4)  \big)  \bigg)
\\&\leq& \frac{k\cdot U_{S,k-4}}{k-4} \cdot \bigg( \log i(x,Z_{i})+\log i(Z_{i},\p(x)) \bigg)+\frac{4k}{k-4}
\\&=& \frac{2k\cdot U_{S,k-4}}{k-4} \cdot \log i(x,Z_{i})+\frac{4k}{k-4}.
\end{eqnarray*}
We have  
\begin{eqnarray*}
\frac{\log i(x,\p(x))}{V_{S,k}}-1
&\leq &\S(x,\p(x),k)
\\&= & \sum_{Z_{i}\in Supp(\phi)} D_{Z_{i}}(x,\p(x))+   \sum_{W\in Supp(\phi)^{c}} D_{W}(x,\p(x))  
\\&\leq & \sum_{Z_{i}\in Supp(\phi)} D_{Z_{i}}(x,\p(x))+ \sum_{i} \bigg( \sum_{W\in Supp(\phi)^{c}_{i}} D_{W}(x,\p(x))  \bigg)
\\&\leq & \sum_{Z_{i}\in Supp(\phi)} D_{Z_{i}}(x,\p(x))+  \frac{2k\cdot U_{S,k-4}}{k-4} \cdot  \bigg(\sum _{i}  \log i(x,Z_{i}) \bigg)+ \frac{4k\cdot \xi(S)}{k-4}.
\end{eqnarray*}                       
\end{proof}

\bibliographystyle{plain}
\bibliography{references.bib}

\begin{thebibliography}{10}

\bibitem{BEH}
Jason~A. Behrstock.
\newblock Asymptotic geometry of the mapping class group and {T}eichm\"uller
  space.
\newblock {\em Geom. Topol.}, 10:1523--1578, 2006.

\bibitem{CR}
Young-Eun Choi and Kasra Rafi.
\newblock Comparison between {T}eichm\"uller and {L}ipschitz metrics.
\newblock {\em J. Lond. Math. Soc. (2)}, 76(3):739--756, 2007.

\bibitem{GadreTsai}
Vaibhav Gadre and Chia-Yen Tsai.
\newblock Minimal pseudo-{A}nosov translation lengths on the complex of curves.
\newblock {\em Geom. Topol.}, 15(3):1297--1312, 2011.

\bibitem{HEM}
John Hempel.
\newblock 3-manifolds as viewed from the curve complex.
\newblock {\em Topology}, 40(3):631--657, 2001.

\bibitem{GGV}
Nicholas G.~Vlamis Jonah~Gaster, Joshua Evan~Greene.
\newblock Coloring curves on surfaces.
\newblock {\em arXiv:1608.01589}.

\bibitem{LIC}
W.~B.~R. Lickorish.
\newblock A representation of orientable combinatorial {$3$}-manifolds.
\newblock {\em Ann. of Math. (2)}, 76:531--540, 1962.

\bibitem{MM2}
H.~A. Masur and Y.~N. Minsky.
\newblock Geometry of the complex of curves. {II}. {H}ierarchical structure.
\newblock {\em Geom. Funct. Anal.}, 10(4):902--974, 2000.

\bibitem{MIN}
Yair~N. Minsky.
\newblock Bounded geometry for {K}leinian groups.
\newblock {\em Invent. Math.}, 146(1):143--192, 2001.

\bibitem{RAF}
Kasra Rafi.
\newblock A combinatorial model for the {T}eichm\"uller metric.
\newblock {\em Geom. Funct. Anal.}, 17(3):936--959, 2007.

\bibitem{ABG}
Jonah~Gaster Tarik~Aougab, Ian~Biringer.
\newblock Determining the finite subgraphs of curve graphs.
\newblock {\em arXiv:1702.04757}.

\bibitem{W4}
Yohsuke Watanabe.
\newblock Intersection numbers in the curve graph with a uniform constant.
\newblock {\em Topology Appl.}, 204:157--167, 2016.

\bibitem{W3}
Yohsuke Watanabe.
\newblock Intersection numbers in the curve complex via subsurface projections.
\newblock {\em J. Topol. Anal.}, 9(3):419--439, 2017.

\bibitem{WEB1}
Richard C.~H. Webb.
\newblock Uniform bounds for bounded geodesic image theorems.
\newblock {\em J. Reine Angew. Math.}, 709:219--228, 2015.

\end{thebibliography}

\end{document}